\documentclass{amsart}
\usepackage{enumerate}

\newcommand{\interior}[1]{\ensuremath{\mathrm{Int}}(#1)}

\def \N{\mathbb{N}}

\def \bR{\mathbb{R}}
\def \cB{\mathcal{B}}

\def \bN{\mathbb{N}}
\def \bZ{\mathbb{Z}}

\DeclareMathOperator{\Int}{int}

\newtheorem{Th}{Theorem}[section]
\newtheorem{Thm}[Th]{Theorem}

\newtheorem*{Def}{Definition}

\newtheorem{Cor}[Th]{Corollary}
\newtheorem{Prop}[Th]{Proposition}

\newtheorem{Lem}[Th]{Lemma}

\newtheorem*{Lem*}{Lemma}

\newtheorem*{thmA}{Theorem A}
\newtheorem*{thmB}{Theorem B}
\newtheorem*{thmC}{Theorem C}
\newtheorem*{thmD}{Theorem D}
\newtheorem*{thmE}{Theorem E}

\begin{document}
\title[Interpreting the MSO theory of $(\bN,+1)$ in expansions of $(\bR,<)$]{Interpreting the monadic second order theory of one successor in expansions of the real line}
\thanks{The first author was partially supported by NSF grant DMS-1300402. The second author was supported by the European Research Council under the European Union's Seventh Framework Programme (FP7/2007-2013) / ERC Grant agreement no.\ 291111/ MODAG}


\author{Philipp Hieronymi}
\address
{Department of Mathematics\\University of Illinois at Urbana-Champaign\\1409 West Green Street\\Urbana, IL 61801}
\email{phierony@illinois.edu}
\urladdr{http://www.math.uiuc.edu/\textasciitilde phierony}

\author{Erik Walsberg}
\address
{Department of Mathematics\\University of Illinois at Urbana-Champaign\\1409 West Green Street\\Urbana, IL 61801}
\email{erikw@math.ucla.edu}
\date{\today}
\begin{abstract}
We give sufficient conditions for a first order expansion of the real line to define the standard model of the monadic second order theory of one successor. Such an expansion does not satisfy any of the combinatorial tameness properties defined by Shelah, such as NIP or even $\textrm{NTP}_2$. We use this to deduce the first general results about definable sets in $\textrm{NTP}_2$ expansions of $(\bR,<,+)$.
\end{abstract}
\maketitle

The goal of this paper is to bring together the study of combinatorial tameness properties of first order structures initiated by Morley and Shelah (\emph{neostability}) and the study of geometric tameness properties of expansions of the real line $(\bR,<)$ championed by Miller (\emph{tame geometry}, see \cite{Miller-tame}). Let $\mathcal B$ be the two-sorted (first order) structure $(\mathcal{P}(\mathbb{N}), \mathbb{N}, \in,+1)$ where $\mathcal{P}(\mathbb{N})$ is the power set of $\mathbb{N}$ and $+1$ is the successor function on $\mathbb N$. The theory of this structure is essentially the monadic second order theory of $(\bN,+1)$. While B{\"u}chi showed in his landmark paper~\cite{buchi} that $\cB$ admits quantifier elimination in a suitable language and its theory is decidable, $\cB$ obviously does not enjoy any Shelah-style combinatorial tameness properties, such as NIP or $\textrm{NTP}_2$ (see e.g. Simon \cite{Simon-Book} for definitions). Therefore any structure that defines an isomorphic copy of $\cB$, can not satisfy any of those properties, and has to be considered complicated or \emph{wild} in this framework of combinatorial tameness. Here we study the consequences of the non-definability of a copy of $\cB$ in an expansion of $(\bR,<)$ on the geometric tameness of definable sets in this expansion. Our results are new even when the assumption \emph{``does not define an isomorphic copy of $\cB$''} is replaced by one of the stronger assumptions \emph{``has $\textrm{NTP}_2$''} or even \emph{``has NIP''}. Therefore these are arguably the very first general results about $\textrm{NTP}_2$ expansions of $(\bR,<,+)$.\\

\noindent Throughout definable will always mean definable with parameters. We will say that a structure \textbf{defines $\cB$} if it defines an isomorphic copy of $\cB$.  The main technical result of this paper is as follows.
\begin{thmA}
Let $\mathcal R = (\mathbb{R},<, D, \prec)$ where $D \subseteq \mathbb{R}$ is dense in some open interval and $\prec$ is an order on $D$ with order type $\omega$.
Then $\mathcal{R}$ defines $\mathcal{B}$.
\end{thmA}
\noindent \emph{Heuristically, an expansion of $(\bR,<)$ can only satisfy any kind of Shelah-style tameness if there is no dense set that can be definably ordered with order type $\omega$.} We envision Theorem A to be a starting point for a forthcoming study of the connections between the two notions of tameness described above. We will now state the main consequences of Theorem A established in this paper.

%In this paper we show that various expansions of $(\bR,<)$ interpret $\cB$.
%This yields new results about expansions of $(\bR,<,+)$ which satisfy Shelah-style combinatorial tameness properties.
%In particular, we prove the first general results about $NTP_2$ expansions of $(\bR,<,+)$.

\subsection*{Cantor sets} We first consider expansions of $(\bR,<,+)$ that define a Cantor set. A \textbf{Cantor set} is a subset of $\bR$ that is nonempty, compact, and has neither isolated nor interior points. Such expansions were first studied by Friedman, Miller, Kurdyka and Speisseger \cite{FKMS}.

\begin{thmB} An expansion of $(\bR,<,+)$ that defines a Cantor set, defines $\cB$.
\end{thmB}
\noindent \emph{Heuristically, an expansion of $(\bR,<,+)$ can only satisfy any kind of Shelah-style tameness if it does not define a Cantor set.} We do not know whether the converse implication of Theorem B is true as well. However, by Boigelot, Rassart and Wolper ~\cite[Theorem 5]{BRW}, $\cB$ defines an isomorphic copy of $(\bR,<,+,C)$ where $C$ is the usual Cantor set constructed by repeatedly deleting middle-thirds of a line segment. Therefore, if an expansion of $(\bR,<,+)$ defines $\cB$, it also defines an isomorphic copy of an expansion of $(\bR,<,+)$ that defines a Cantor set. Also note that Theorem B fails for expansions of $(\bR,<)$ by Dolich, Miller and Steinhorn \cite[7.1]{dms}.\\

\noindent This result is also significant as a complement to ~\cite{Cantor}. There a Cantor set $K$ is constructed such that every $(\bR,<,+,\cdot,K)$-definable set is Borel. The key idea is to find a Cantor set $K$ such that $(\bR,<,+,\cdot,K)$ defines a model of $\cB$ and essentially nothing else, and to use known results on $\cB$ to prove a quantifier elimination result for $(\bR,<,+,\cdot,K)$. Theorem B arguably shows that one must use results about $\cB$ to prove model theoretic tameness results for expansions which define Cantor sets.

%There are a number model-theoretically tame first order expansions of $(\bR,<,+)$ which are known to interpret $\cB$.
%There are cantor sets $K \subseteq \bR$ such that $(\bR,<,+,\times,K)$ is decidible and is known to interpret $\mathcal{B}$~\cite{Cantor}.
%If $\alpha \in \bR$ is a quadratic imaginary then $(\bR,<,+,\mathbb{Z},\mathbb{Z}\alpha)$ is bi-interpretable with $\cB$~\cite{subgroup}.
%Let $r$ be a natural number strictly greater then $2$.
%Let $\mathcal{T}_r$ be the expansion of $(\mathbb{R},<,+,\mathbb{Z})$ by a ternary predicate $V_r(x,u,k)$ which holds if and only if $u$ is a positive integer power of $r$, $k \in \{0,\ldots, r - 1\}$ and the $u$th digit of the %base-r representation of $x$ is $k$.
%It is easy to see that $\mathcal{T}_r$ interprets $\cB$.
%Boigelot, Rassart and Wolper~\cite{BRW} showed that the theory of $\mathcal{T}_r$ is decidable.
%These examples motivated the results of this paper.
%The key result is the following:

\subsection*{Collapsing discrete sets} By \cite[Theorem A]{discrete2} an expansion of $(\bR,<,+,\cdot)$ that defines a discrete set $D \subseteq \bR$ and a function $f: D^k \to \bR$ whose image is somewhere dense, also defines $\mathbb{Z}$. The theory of such an expansion is clearly undecidable, and moreover defines every real projective set in the sense of descriptive set theory (see e.g. Kechris \cite[37.6]{kechris}). Therefore such structures are maximally wild from the point of view of first order model theory. In contrast, there are expansions of $(\bR,<,+)$ that define a discrete set $D \subseteq \bR$ and a function $f: D^k \to \bR$ whose image is somewhere dense, but whose theory is decidable. For example, by \cite{H-Twosubgroups} $(\bR,<,+,\mathbb Z,\sqrt{2}\mathbb Z)$ is such a structure. All known examples of such structures define $\cB$. We show that this is necessairly the case.

%and Theorem~\ref{disctodense} below:
\begin{thmC}
Let $\mathfrak{R}$ be an expansion of $(\bR,<,+)$ that defines a discrete set $D \subseteq \bR^k$ and a map $f: D \to \bR$ whose image is somewhere dense.
Then $\mathfrak{R}$ defines $\cB$.
\end{thmC}
\noindent  \emph{Heuristically, an expansion of $(\bR,<,+)$ that definably collapses a discrete set, can not satisfy any kind of Shelah-style tameness.} The converse implication of Theorem C is not true. The expansion of the real field by the Cantor set $K$ constructed in ~\cite{Cantor} defines $\cB$ and does not satisfy the assumptions of Theorem C. However, $\cB$ defines an isomorphic copy of $(\bR,<,+,\mathbb Z,\sqrt{2}\mathbb Z)$ by ~\cite[Theorem C]{H-Twosubgroups}. For that reason, an expansion of $(\bR,<,+)$ that defines $\cB$, also defines an isomorphic copy an expansion of $(\bR,<,+)$ that satisfies the assumptions of Theorem C.\\

\noindent Dolich and Goodrick recently established Theorem C when \emph{``defines $\cB$''} is replaced by \emph{``is not strong''} (see ~\cite[2.17]{DG}). We refer the reader to their paper for a definition of strongness, but note that strong structures form a subclass of all $\textrm{NTP}_2$ structures. Theorem C can be regarded as a significant generalization of their result for expansions of $(\bR,<,+)$.

\subsection*{Dimension equality}  In \cite{HM} Hieronymi and Miller showed that various notions of metric and topological dimensions coincide on closed definable sets in expansions of the real field that do not define $\mathbb Z$. In particular, topological dimension and Assouad dimension are equal on such sets. We refer the reader to \cite{HM} for a precise statements and definitions of the dimensions involved. An expansion of the real field that does not define $\cB$ cannot define $\mathbb Z$. It is therefore natural to ask if non-definability of $\cB$ has stronger consequences on the dimensions of definable sets.

\begin{thmD}
Let $\mathcal{R}$ be an expansion of $(\mathbb{R},<,+,\cdot)$ that does not define $\mathcal B$ and does not define a dense and co-dense subset of $\mathbb{R}$.
If $A \subseteq \mathbb{R}^k$ is definable in $\mathcal R$, then the topological dimension and the Assouad dimension of $A$ agree.
\end{thmD}
\begin{proof}
It follows directly from \cite[1.6]{HM} that topological dimension and Assouad dimension agree on definable sets in an expansion of $(\mathbb{R},<,+,\cdot)$ that does not define a Cantor subset of $\mathbb{R}$ or a dense and co-dense subset of $\mathbb{R}$. Now apply Theorem B.
\end{proof}

\subsection*{Tame topology} The results in this paper can be used to develop tame topology\footnote{Tame topology in sense of van den Dries \cite{tametop}.} for expansions of $(\bR,<,+)$ that do not define $\cB$. To show the viability of such a project, we give a weak generalization of the monotonicity theorem for o-minimal expansions of $(\bR,<,+)$. A function $f : \bR\to \bR$ is \textbf{Lipschitz} on $X\subseteq \bR$ if there is a $\lambda \in \bR_{>1}$ such that $| f(x) - f(y) | \leq \lambda | x - y |$ for all $x,y \in X$, $f$ is \textbf{bi-Lipschitz} on $X$ if there is a $\lambda \in \bR_{>}$ such that
\[
\frac{1}{\lambda}|x-y| \leq |f(x)-f(y)| \leq \lambda |x-y| \quad \text{for all } x,y \in X.
\]
\begin{thmE}
Let $\mathcal R$ be an expansion of $(\bR,<,+)$ that does not define $\cB$. Let $f: \bR \to \bR$ be a continuous $\mathcal R$-definable function. Then there is a definable open dense $U \subseteq \bR$ such that $f$ is strictly increasing, strictly decreasing, or constant on every connected component of $U$. Moreover, there is an open dense subset $V$ of $\bR$ such that $f$ is either constant or locally bi-Lipschitz on each connected component of $V$.
\end{thmE}

\noindent Fornasiero \cite{Forna} proved Theorem E for expansions of $(\bR,<,+,\cdot)$ that do not define $\bZ$. Theorem E is new for NIP expansions of $(\bR,<,+)$. While proving Theorem E, we also show that
every subset of $\bR$ definable in an expansion of $(\bR,<,+)$ that does not define $\cB$, is either somewhere dense or has Lebesgue measure zero.

\subsection*{Acknowledgements} The authors would like to thank Chris Miller for helpful comments on an earlier version, William Balderrama for spotting several typos and the anonymous referee for carefully reading this paper.

\subsection*{Notations and Terminology}
Given a linear order $(D,\prec)$ and $d \in D$ we let $D_{\prec d} = \{e \in D : e \prec d\}$ and $D_{\preceq d} = \{e \in D : e \preceq d\}$.
We let $\bR^>$ be the set of positive reals and $\bR^\geq$ be the set of nonnegative reals.
Let $\mathcal{A} = \{A_t : t > 0\}$ be a family of sets.
We say that $\mathcal{A}$ is \textbf{increasing} if $A_t \subseteq A_{t'}$ whenever $t < t'$ and we say that $\mathcal{A}$ is \textbf{decreasing} if $A_{t} \subseteq A_{t'}$ when $t' < t$.
Let $K \subseteq \mathbb{R}$ be closed.
The connected components of $\mathbb{R} \setminus K$ are open intervals, we refer to them as the \textbf{complementary intervals} of $K$.
Given $A \subseteq B \subseteq \bR$ and $\delta > 0$ we say that $A$ is \textbf{$\delta$-dense} in $B$ if for every $b \in B$ there is an $a \in A$ such that $|a - b| < \delta$.
The interior of a set $A\subseteq \bR$ is denote by $\Int(A)$.

\section{Proof of Theorem A}
Let $\mathcal R = (\mathbb{R},<, D, \prec)$ where $D$ is a dense subset of an open interval of $\mathbb{R}$ and $\prec$ is an order on $D$ with order type $\omega$. In this section, we will show that $\mathcal{R}$ defines $\mathcal{B}$. The overall idea of this proof is based on earlier work of Hieronymi and Tychonievich~\cite{HT}. Recall that a partial order has order type $\omega$ if and only if it is linear and every initial segment is finite. We suppose without loss of generality that $D$ is dense in $[0,1]$, $0,1 \in D$, $1$ is the $\prec$-minimum of $D$ and $0$ is the $\prec$-successor of $1$.

\begin{Def} Let $d_1,d_2 \in D$ be such that $d_1 \prec d_2$.
We say that $D$ \textbf{splits} between $d_1$ and $d_2$ if for every $e_1,e_2 \in D$ such that $e_1 < e_2$ and $e_1,e_2 \preceq d_1$, there is an $e_3 \in D$ such that $d_1 \prec e_3 \prec d_2$ and $e_1 < e_3 < e_2$.
We say that $E \subseteq D$ is a \textbf{splitting set} if $E$ is infinite and $D$ splits between every $d_1,d_2 \in E$ with $d_1 \prec d_2$.
\end{Def}
\noindent Since $D$ is $\leq$-dense in $[0,1]$, we can find for every $d_1 \in D$ an element $d_2 \in D$ such that $D$ splits between $d_1$ and $d_2$.

\begin{Def} Let $c \in [0,1]$ and $d \in D$. We say that \textbf{the best approximation of $c$ from left before $d$}, denoted by $l(c,d)$, is the $<$-maximal element of $D_{\preceq d}$ in $[0,c]$. \textbf{The best approximation of $c$ from right before $d$}, denoted by $r(c,d)$, is the $<$-minimal element of $D_{\preceq d}$ in $[c,1]$. We denote the set of all best approximation of $c$ from the left by $L_c$.
\end{Def}

\noindent When $d \succeq 0$, then $l(c,d)$ and $r(c,d)$ exists for every $c\in [0,1]$. Note that for every $c\in [0,1]$ and every  $d,e \in D$ with $d\preceq e$, we have $l(d,e)=d$ and $l(l(c,e),d)=l(c,d)$.

\begin{Lem}
There is a definable splitting set.
\end{Lem}
\begin{proof}
We inductively construct a sequence $\{e_i\}_{i \in \bN}$ of elements of $D$ which is both strictly $\prec$-increasing and $\leq$-increasing such that for every $n$ there is an $\epsilon_n > 0$ such that for all $c \in (e_n,e_n + \epsilon_n)$:
\begin{enumerate}
\item $L_c \cap D_{\preceq e_n} = \{e_0,\ldots,e_n\}$.
\item $D$ splits between $e_i$ and $e_{i+1}$  for all $0 \leq i \leq n - 1$.
\end{enumerate}
For the base case we take $e_0$ to be $0$ and $\epsilon_0 = 1$.
We now treat the inductive case.
We assume that we have $e_0,\ldots,e_n$ and $\epsilon_n$ satisfying the conditions and construct $e_{n+1}$ and $\epsilon_{n+1}$.
Let $d \in D$ be such that $e_n \prec d$ and $D$ splits between $e_n$ and $d$.
Let $0 < \delta < \epsilon_n$ be such that $(e_n,e_n+\delta) \cap D_{\preceq d} = \emptyset$.
This is possible as $D_{\preceq d}$ is finite.
We let $e_{n+1}$ be the $\preceq$-least element of $D \cap (e_n, e_n + \delta)$.
This implies that $d \prec e_{n+1}$, it follows that $D$ splits between $e_n$ and $e_{n+1}$.
We now let $\epsilon_{n+1} > 0$ be such that $e_{n+1} + \epsilon_{n+1} = e_n + \delta$.
If $c \in (e_{n+1}, e_{n+1} + \epsilon_{n+1})$, then $e_{n+1} \in L_c$. If $e_n \prec e \prec e_{n+1}$, then $e > c$ so $e \notin L_c$.
We have
\[
(e_{n+1}, e_{n+1} + \epsilon_{n+1}) \subseteq (e_n, e_n + \delta) \subseteq (e_n, e_n + \epsilon_n).
\]
For all $c \in (e_{n+1}, e_{n+1} + \epsilon_{n+1})$ we have:
\begin{enumerate}
\item $L_c \cap D_{\preceq e_{n+1}} = \{e_0, \ldots, e_{n+1}\}$,
\item $D$ splits between $e_i$ and $e_{i+1}$ for all $0 \leq i \leq n + 1$.
\end{enumerate}
This completes the inductive construction.
Note that $\{e_i\}_{i \in \bN}$ is a splitting set.
We show that $\{e_i\}_{i \in \bN}$ is definable.
As $\{e_i\}_{i \in \bN}$ is $\leq$-bounded and $\leq$-increasing, there is an $a \in [0,1]$ such that $e_i \to a$ as $i \to \infty$.
Then $a \in (e_n, e_n + \epsilon_n)$ for all $n$, so $L_a \cap D_{\preceq e_n} = \{e_0,\ldots,e_n\}$ for all $n$ and thus $L_a = \{e_i\}_{i \in \bN}$.
\end{proof}

\noindent For the following, fix a definable splitting set $E = \{e_i\}_{i \in \bN}$ with $e_i \prec e_j$ whenever $i<j$ and $e_0=0$. Since $(D,\prec)$ has order type $\omega$, so does $(E,\prec)$. We denote the successor function on $(E,\prec)$ by $s$ and the predecessor function by $p$. The function $p$ is only defined on $E_{\succ e_0}$. Since $e_0=0$, we get that $l(c,e_0)=0$ and $r(c,e_0)=1$ for every $c \in (0,1)$. In the isomorphic copy of $\cB$ that we are trying to define, the set $E$ will serve as the copy of $\N$. We will now construct three further sets $A,B,C$ that will be used to define the copy of $\mathcal P(\N)$.

\begin{Def} Define $C \subseteq (0,1)$ to be the set of all $c\in (0,1)$ such that for every $e \in E_{\succ e_0}$ either $l(c,p(e))=l(c,e)$ or $r(c,p(e))=r(c,e)$. Given $c \in [0,1]$ we set
\[
F_c = \{ e \in E_{\succ e_0}  \ : \ l(c,p(e)) \neq l(c,e) \}.
\]
Let $A$ be the set of $c \in C$ such that $F_c$ is neither finite nor co-finite in $E$. Let $B\subseteq D$ be the set of all $d\in D$ such that there is $c \in A$  and $e \in E_{\succ e_0}$ with $d=l(c,e)$.
\end{Def}

\noindent Since $(E,\prec)$ has order type $\omega$, the set $A$ is definable. Note that $F_c \cap E_{\preceq e} = F_{l(c,e)}$ for all $c\in (0,1)$ and $e\in E_{\succ e_0}$. In the isomorphic copy of $\cB$, the set $A$ will correspond to the set of all subsets of $\N$ that are neither finite nor co-finite, while $B$ will correspond to the set of all finite subsets of $\N$.

\begin{Lem}\label{lem:uniqueness} Let $e \in E_{\succ e_0}$ and $a,b \in C\setminus D$. If $F_a\cap E_{\preceq e}= F_b\cap E_{\preceq e}$, then $l(a,e)=l(b,e)$ and $r(a,e)=r(b,e)$.
\end{Lem}
\begin{proof} With loss of generality we may assume $a<b$. Suppose that $l(a,e)\neq l(b,e)$ or $r(a,e)\neq r(b,e)$. Then there is $d \in D_{\preceq e}$ such that $a< d< b$. Since $(E,\prec)$ has order type $\omega$, we can assume that $e$ is $\prec$-minimal in $E_{\succ e_0}$ with this property. Together with $l(a,e_0)=l(b,e_0)=0$ and $r(a,e_0)=r(b,e_0)=1$, this gives
\begin{equation}\label{lem:eq1}
l(a,p(e))=l(b,p(e)) \hbox{ and } r(a,p(e))=r(b,p(e)).
\end{equation}
Thus $p(e) \prec d \preceq e$. But then
\begin{equation}\label{lem:eq2}
l(a,e)\neq l(b,e) \hbox{ and } r(a,e)\neq r(b,e).
\end{equation}
Since $F_a\cap E_{\preceq e}= F_b\cap E_{\preceq e}$, we get that either $l(a,p(e))=l(a,e)$ and $l(b,p(e))=l(b,e)$, or $r(a,p(e))=r(a,e)$ and $r(b,p(e))=r(b,e)$. This contradicts \eqref{lem:eq1} and \eqref{lem:eq2}.
\end{proof}

\begin{Lem}\label{lem:exists} For $X \subseteq E_{\succ e_0}$ that is neither finite nor co-finite, there is unique $a\in C$ such that $F_a=X$.
\end{Lem}
\begin{proof} The uniqueness follows immediately from Lemma \ref{lem:uniqueness} and the density of $D$. It is left to show the existence of such an $a$. We will construct a nested sequence $\{I _n \}_{n \in \bN}$ of non-empty closed subintervals of $[0,1]$ such that for every $c \in \Int(I_n)$
\begin{itemize}
\item [(I)] $I_n = [l(c,e_n),r(c,e_n)]$,
\end{itemize}
and if $n > 0$,
\begin{itemize}
\item [(II)] $l(c,e_{n-1}) \neq l(c,e_n)$ if and only if $e_n \in X$,
\item [(III)] $l(c,e_{n-1})=l(c,e_n)$ or $r(c,e_{n-1})=r(c,e_n)$.
\end{itemize}

\noindent Let $I_0=[0,1]$. Since $e_0=0$, $I_0$ satisfies (I). Now suppose that we have already constructed an interval $I_{n - 1}$ satisfying (I)-(III). By (I) there are $d_1,d_2 \in D_{\preceq e_{n - 1}}$ such that $I_{n-1} = [d_1,d_2]$.\\

\noindent First suppose that $e_n \in X$. Let $d_3$ be the $\leq$-maximal element of $D_{\preceq e_{n}} \cap (d_1,d_2)$. Since $E$ is a splitting set, such that a $d_3$ exists. By (I) for $I_{n-1}$, we necessarily have $d_2 \prec d_3$.
If $c \in (d_3,d_2)$, then it is easy to see that $l(c,e_{n}) = d_3$. For such $c$, $l(c,e_{n-1}) \neq  l(c,e_{n})$. Since $d_3$ was chosen to be the $\leq$-maximal element of $D_{\preceq e_{n}} \cap (d_1,d_2)$, we also get that $r(c,e_n)=r(c,e_{n-1})=d_2$ for all $c \in (d_3,d_2)$. Therefore set $I_{n}= [d_3,d_2]$.\\

\noindent Now consider the case that $e_n \notin X$. Let $d_3$ be the $\leq$-minimal element of $D_{\preceq e_{n}} \cap (d_1,d_2)$. Since $E$ is a splitting set, such that a $d_3$ exists, and $d_1 \prec d_3$ by (I) for $I_{n-1}$.
If $c \in (d_1,d_3)$, then it is easy to see that $r(c,e_{n}) = d_3$. Since $d_3$ was chosen to be the $\leq$-minimal element of $D_{\preceq e_{n}} \cap (d_1,d_2)$, we also get that $l(c,e_{n-1})=l(c,e_n)=d_1$ for such $c$.
Therefore set $I_{n}= [d_1,d_3]$.\\

\noindent It follows directly from the construction that $I_n$ satisfies (I)-(III). By completeness of $\bR$, $\bigcap_{n\in \N} I_n\neq \emptyset$. Let $a\in \bigcap_{n\in \N} I_n$.
In order to show that $F_a=X$ and $a\in C$, it is by (II) and (III) enough to show that $a\in \Int(I_n)$ for every $n\in \N$. Suppose not. Then there is a minimal $m\in \N$ such that $a=l(c,e_m)$ or $a=r(c,e_m)$ for every $c \in \Int(I_m)$.
Let us consider the case that $a=l(c,e_m)$. Since the sequence of intervals is nested, $a$ has to be the left endpoint of every interval $I_n$ for $n\geq m$. Therefore, $a = l(c,e_n)$ for every $c \in \Int(I_n)$ and every $m\geq n$. By (II), $e_n\notin X$ for every $n\geq m$. This contradicts the assumption that $X$ is not finite. In the case that $a=r(c,e_m)$, one can reach a contradiction against the assumption that $X$ is not co-finite in a similar way.
\end{proof}

\begin{Cor}\label{cor:exists} For a finite $X \subseteq E_{\succ e_0}$ there is a unique $d\in B$ with $X = F_d$.
\end{Cor}
\begin{proof} We first show the existence of such $d$. Suppose $X \subseteq E_{\succ e_0}$ is finite. Take $Y\subseteq E_{\succ e_0}$ that contains $X$ as an initial segment and is neither finite or co-finite. By Lemma \ref{lem:exists} there is $c \in C$ such that $F_c = Y$. Let $e \in E$ be the $\prec$-maximum of $X$. Then $X=F_{c} \cap E_{\preceq e} = F_{l(c,e)}$.\\

\noindent It is left to prove uniqueness. Suppose there is another $d' \in B$ such that $F_{d'} = X$. Let $c'\in A$ and $e' \in E$ such that $l(c',e')=d'$. Since $e$ is the $\prec$-maximum of $X$, $e'\succeq e$. Since $F_{l(c',e')}=X$, there is no $e'' \in F_{c'}$ with $e \prec e'' \preceq e'$. Thus $l(c',e')=l(c',e)$. Since $F_{c} \cap E_{\preceq e} = X = F_{c'} \cap E_{\preceq e}$, we have $l(c',e) = l(c,e)$ by Lemma \ref{lem:uniqueness}. Thus $d=d'$.
\end{proof}

\begin{proof}[Proof of Theorem A]
We will now construct the isomorphic copy of $\cB$. Let $P := \{0\} \times A \cup \{0,1 \} \times B$. For $p=(p_1,p_2)\in P$ and $e \in E_{\succ e_0}$ we say $\epsilon(e,p)$ holds whenever either $p_1=0$ and $e \in F_{p_2}$, or $p_1=1$ and $e \notin F_{p_2}$. By Lemma \ref{lem:exists} and Corollary \ref{cor:exists} $(E_{\succ e_0},P,\epsilon,s)$ is an isomorphic copy of $\cB$.
\end{proof}

%Fix a function $g: E \to \{0,1\}$.
%We apply induction to prove the following claim:
%\begin{clm}
%There is a nested sequence of $\{I _n \}_{n \in \bN}$ of open subintervals of $(0,1)$ such that if $c \in I_n \setminus D$ then
%$$ f_c(d_i) = g(d_i) \quad \text{for all } 0 \leq i \leq n. $$
%\end{clm}
%\-Each $I_n$ will have endpoints in $D_{\prec d_n}$.
%We suppose that we have an interval $I_{n - 1} = (d,d')$ with $d,d' \in D_{\prec d_{n - 1}}$ satisfying the claim for $n - 1$.
%We first suppose that $g(n) = 1$.
%Let $e$ be the $\leq$-minimal element of $D_{\prec d_{n}} \cap (d,d')$.
%We necessarily have $d \prec e$.
%Thus if $c \in (d,e)$ $\alpha(c,d_{n}) = d$ and so $f_c(d_{n+1}) = 1$.
%In this case we declare $I_{n} = (d,e)$.
%Now suppose that $g(n) = 0$.
%Let $e$ be the $\leq$-maximal element of $D_{\preceq d_{n}} \cap (d,d')$.
%We necessarily have $d' \prec e$.
%If $c \in (e,d')$ then $\alpha(c,d_{n}) = e$ and so $f_c(d_{n + 1}) = 0$.
%In this case we declare $I_{n} = (e,d')$.

%Let $c$ be such that $c \in I_n$ for all $n$.
%By construction, if $I_n$ does not intersect $D_{\prec d_n}$.
%It follows that $c \notin D$ so $f_c$ is defined.
%For this $c$ we have $f_c = g$.

\section{Proof of Theorem B}
Let $K \subseteq \bR$ be a Cantor set; that is $K$ is a subset of $\bR$ that is nonempty, compact, and has neither interior nor isolated points. We show that $(\bR,<,+,K)$ defines $\cB$. In this section \emph{``definable''} means \emph{``definable in $(\bR,<,+,K)$''}. Let $C \subseteq K$ be the set of right endpoints of bounded complementary intervals of $K$. In between any two complementary intervals there is a third. Therefore $(C,<)$ is a countable dense linear order without endpoints and thus has order type $(\mathbb{Q},<)$. Let $L\subseteq K$ be the set of all elements of $K$ that are not in  any closure of a complementary interval of $K$. Both $C$ and $L$ are definable. Note that $(C\cup L,<)$ is a dense linear order without endpoints that has the least-upper-bound property\footnote{A linear order $(X,\prec)$ has the \textbf{least-upper-bound property} if every non-empty subset of $X$ with an upper bound has a supremum in $X$.}. Moreover, $C\cup L$ has a countable dense subset, namely $C$. A linear order with these properties is order-isomorphic to $(\bR,<)$ (see e.g. Jech \cite[p.32]{Jech}). Thus $(C\cup L,<)$ is order-isomorphic to $(\bR,<)$.
%We put an equivalence relation $\sim$ on $(0,1)$ by declaring $a \sim b$ if and only if $C \cap (0,a] = C \cap (0,b]$. Note that $a\sim b$ if and only if
%\begin{itemize}
%\item $a,b$ are in the closure of the same complementary interval of $K$, or
%\item $a,b\in L$ and $a=b$.
%\end{itemize}
%Therefore $\sim$ is definable and $C\cup L$ is complete set of representatives of $\sim$. As $\sim$-equivalence classes are convex, $\leq$ pushes forward to a linear order on the quotient $P := [0,1]/_{\sim}$ which we also %denote by $\leq$. Then $(P,<)$ is the completion in the sense of linear orders of $(C,\leq)$ and is thus order-isomorphic to $(\bR,<)$. Since $C\cup L$ is a complete set of representatives, $(C\cup L,<)$ is order-isomorphic to %$(\bR,<)$.
By Theorem A and since $C$ is dense in $C\cup L$, it is only left to put a definable order $\prec$ on $C$ with order type $\omega$.
We let $\tau: C \to \bR$ be the definable function that maps the right endpoint of each complementary interval of $K$ to the length of that interval.
Note that for each $\delta > 0$ there are only finitely many $c \in C$ such that $\tau(c) \geq \delta$.
Let $a, b \in C$.
We declare $a \prec b$ if $\tau(a) > \tau(b)$ or if $\tau(a) = \tau(b)$ and $a > b$.
This is a linear order, and if $a \prec b$, then $\tau(a) \geq \tau(b)$. So each initial segment of $(C,\prec)$ has only finitely many elements.
Thus $(C, \prec)$ has order type $\omega$.
%The following is a corollary to Theorem~\ref{cantor} and the results of \cite{HM}.
%The Assouad dimension is a notion of dimension in geometric measure theory.
%See [ref] for a definition.
%It bounds Minkowski dimension and the Hausdorff dimension from above.
%\begin{Thm}
%Let $\mathcal{M}$ be an expansion of $(\mathbb{R},<,+,\times)$ which does not interpret $\mathcal B$ and which does not define a dense and co-dense subset of $\mathbb{R}$.
%If $A \subseteq \mathbb{R}^k$ is definable then the topological dimension and the Assouad dimension of $A$ agree.
%\end{Thm}
%Note in particular that this theorem applies to any $NTP_2$ expansion of $(\mathbb{R},<,+,\times)$ which does not define a dense and co-dense subset of $\mathbb{R}$.

%The following forms a weak converse to Theorem~\ref{cantor}
%\begin{Prop}
%The following are equivalent for a first order structure $\mathfrak{R}$.
%\begin{enumerate}
%\item $\mathfrak{R}$ interprets $\mathcal{B}$.
%\item $\mathfrak{R}$ interprets $(\bR,<,+,K)$ for some Cantor set $K \subseteq \bR$.
%\item $\mathfrak{R}$ interprets $(\bR,<,+,C)$ where $C$ is the classical Cantor set.
%\end{enumerate}
%\end{Prop}
%\begin{proof}
%It follows immediately from Theorem~\ref{cantor} that if a first order structure interprets $(\mathbb{R},<,+,K)$ for some Cantor set $K$ then it interprets $\mathcal{B}$.
%It therefore suffices to show that $(1)$ implies $(3)$.
%\end{proof}

\section{Proof of Theorem C}
For the remainder of this paper $\mathfrak{R}$ is an expansion of $(\bR,<,+)$ that does not define $\cB$, and \emph{``definable''} means \emph{``definable in $\mathfrak R$''}.  The goal of this section is to show that there is no definable map $f: D \to \bR$ such that $D \subseteq \bR^k$ is discrete and $f(D)$ is somewhere dense.

\begin{Lem}\label{family1}
Let $\mathcal{A} = \{A_t : t > 0\}$ be a definable family of finite subsets of $\bR$ which is either increasing or decreasing.
Then the union of $\mathcal{A}$ is nowhere dense.
\end{Lem}
\begin{proof}
We suppose that $\mathcal{A}$ is decreasing.
A slight modification of the argument may be used to prove the increasing case.
We show that $D := \bigcup_{t > 0} A_t$ admits a definable linear order $\prec$ with order type $\omega$.
The statement then follows by applying Theorem A.
Let $\tau : D \to \bR \cup \{\infty\}$ be given by $ \tau(a) := \sup\{t > 0 : a \in A_t\}$.
Given $a, a' \in D$ we declare $a \prec a'$ if $\tau(a) > \tau(a')$ or if $\tau(a) = \tau(a')$ and $a < a'$.
Then $(D, \prec)$ is linear.
If $a \in D$, then $D_{\prec a}$ is a subset of $A_t$ for any $t < \tau(a)$.
As every $A_t$ is finite, this implies that $D_{\prec a}$ is finite for every $a \in A$.
Thus $(D,\prec)$ has order type $\omega$.
\end{proof}
\noindent We first prove a simple lemma:
\begin{Lem}\label{core}
Let $\mathcal M$ be an expansion of $(\bR,<,+)$.
One of the following holds:
\begin{enumerate}
\item Every bounded nowhere dense definable subset of $\mathbb{R}$ is finite.
\item There is a bounded discrete definable $D \subseteq \mathbb{R}^>$ whose closure is $D \cup \{0\}$.
\end{enumerate}
\end{Lem}

\begin{proof}
Suppose that $(1)$ above does not hold.
Let $C$ be a bounded infinite nowhere dense subset of $\mathbb{R}$.
After replacing $C$ with its closure if necessary we assume that $C$ is closed.
We let $D \subseteq \mathbb{R}^>$ be the set of the lengths of the bounded complementary intervals of $C$.
As $C$ is bounded, for each $\delta > 0$ there are only finitely many complementary intervals with length at least $\delta$.
This implies that $D$ is bounded and discrete and that $D$ has no limit points other then $0$.
As $C$ is infinite and nowhere dense it has infinitely many complementary intervals, and therefore for each $\delta > 0$ there is a complementary interval of length at most $\delta$.
This implies that $0$ is a limit point of $D$.
\end{proof}

\begin{Lem}\label{family2}
Let $\mathcal{A} = \{A_t : t > 0\}$ be an increasing or decreasing definable family of nowhere dense subsets of $\bR$.
Then the union of $\mathcal{A}$ is nowhere dense.
\end{Lem}

\begin{proof}
We apply Lemma~\ref{core}.
If $(1)$ of Lemma~\ref{core} holds then every member of $\mathcal{A}$ is finite and the lemma follows by applying Lemma~\ref{family1}.
We therefore assume that $(2)$ of Lemma~\ref{core} holds and let $D \subseteq \mathbb{R}^>$ satisfy the conditions of $(2)$.

We first suppose that $\mathcal{A}$ is decreasing.
As $\mathcal{A}$ is decreasing the union of $\mathcal{A}$ equals the union of $\{A_t : t \in D\}$.
It suffices to show that the union of $\mathcal{A}$ is not dense in a fixed bounded open interval $I$.
After replacing $\mathcal{A}$ with $\{ A_t \cap I : t > 0 \}$ if necessary we suppose that each $A_t$ is contained in $I$ and therefore bounded.
After replacing each $A_t$ with its closure if necessary we assume that each member of $\mathcal{A}$ is closed.
For each $\delta > 0$ and $t \in D$ we let $C_{\delta,t} \subseteq A_t$ be the set of endpoints of complementary intervals of $A_t$ of length at least $\delta$.
Each $C_{\delta,t}$ is finite.
The set of endpoints of complementary intervals of $A_t$ is dense in $A_t$. Thus for every $\epsilon > 0$ and $t \in D$ there is a $\delta > 0$ such that if $0 < s < \delta$ then $C_{s,t}$ is $\epsilon$-dense in $A_t$.
Given $t \in D$ we let $g(t)$ be the supremum of all $\delta > 0$ such that $C_{\delta,t}$ is $t$-dense in $A_t$.
Given $t \in D$ we declare $B_t = C_{\frac{1}{2}g(t),t}$.
Then $B_t$ is finite and $t$-dense in $A_t$ for all $t \in D$.
We finally declare
$$ F_t =\bigcup_{s \in D, s \geq t} B_s \quad \text{for all } t \in D.  $$
Each $F_t$ is a finite union of finite sets and therefore finite.
Let $F$ be the union of the $F_t$.
The family $\{F_t : t \in D\}$ is decreasing so Lemma~\ref{family1} implies that $F$ is nowhere dense.
We show that $F$ is dense in the union of $\mathcal{A}$.
It follows that the union of $\mathcal{A}$ is nowhere dense.
It suffices to fix $\delta,t > 0$ and show that $F$ is $\delta$-dense in $A_t$.
After decreasing $\delta$ is necessary we suppose that $\delta < t$ and that $\delta \in D$.
Then $B_\delta$ is $\delta$-dense in $A_\delta$.
As $A_t \subseteq A_\delta$ it follows that $B_\delta$ is $\delta$-dense in $A_t$.
As $B_\delta \subseteq F$ if follows that $F$ is $\delta$-dense in $A_t$.

We now treat the increasing case.
We suppose towards a contradiction that the union of $\mathcal{A}$ is dense in an open interval $I$.
Thus for every $\epsilon > 0$ there is a $t > 0$ such that $A_t$ is $\epsilon$-dense in $I$.
Let $f: \bR^> \to \bR^\geq$ be given by letting $f(\epsilon)$ be the infimum of all $t > 0$ such that $A_t$ is $\epsilon$-dense in $I$.
We let $\{ B_t : t \in D\}$ be the definable family given by $B_t = A_{f(t) + 1}$. Note that $\{B_t : t \in D\}$ is now a \emph{decreasing} definable family of nowhere dense sets.
Each $B_t$ is $t$-dense in $I$, so the union of the $B_t$ is dense in $I$. This yields a contradiction to what we proved above.
\end{proof}
\noindent Then next lemma is an immediate consequence of the previous.
\begin{Lem}\label{family3}
Let $\mathcal{A} = \{A_{s,t} : s,t > 0\}$ be a definable of family of nowhere dense subsets of $\bR$ such that:
\begin{enumerate}
\item $A_{s,t} \subseteq A_{s,t'}$ if $t < t'$,
\item $A_{s, t} \subseteq A_{s',t}$ if $s' < s$.
\end{enumerate}
Then the union of $\mathcal{A}$ is nowhere dense.
\end{Lem}
\begin{proof}
For each $s > 0$ we declare: $B_t := \bigcup_{s > 0} A_{s,t}. $ For each $s > 0$ the family $\{ A_{s,t} : t > 0\}$ is increasing, Lemma~\ref{family1} shows that each $B_t$ is nowhere dense.
As $\{B_t : t > 0\}$ forms a decreasing definable family of nowhere dense sets, it follows from another application of Lemma~\ref{family1} that the union of the $B_t$ is nowhere dense.
As the union of $\mathcal{A}$ is the same as the union of the $B_t$, the union of $\mathcal{A}$ is nowhere dense.
\end{proof}

\begin{proof}[Proof of Theorem C]
Let $D \subseteq \bR^k$ be a definable discrete set and $f: D \to \bR$ be a definable function. We need to show that $f(D)$ is nowhere dense. We let $\|\text{ } \|_1$ be the usual $l_1$ norm on $\bR^k$; that is
$ \| (x_1,\ldots, x_k)\|_1 = \sum_{i = 1}^{k} |x_i|. $
We say that an element $d$ of a set $A \subseteq \mathbb{R}^k$ is \textbf{$s$-isolated} if $\| d - d'\|_1 > s$ for any $d' \in A$ such that $d \neq d'$.
Given $s,t > 0$ we let $D_{s,t}$ be the set of $d \in D$ such that $\|d\|_1 \leq t$ and $d$ is $s$-isolated.
Each $D_{s,t}$ is discrete, closed and bounded.
Therefore each $D_{s,t}$ is finite.
Then $D$ is the union of the $D_{s,t}$.
We let $A_{s,t} = f(D_{s,t})$ for all $s,t > 0$.
The image of $f$ agrees with the union of the $A_{s,t}$.
As each $A_{s,t}$ is finite it follows from Lemma~\ref{family3} that the union of the $A_{s,t}$ is nowhere dense.
\end{proof}

\section{Continuous one-variable functions}
In this section we prove generic local monotonicity and Lipschitz results for continuous definable functions. Recall our standing assumption that $\mathfrak{R}$ is an expansion of $(\bR,<,+)$ that does not define $\mathcal{B}$.
We first prove a simple lemma about one-variable functions.
Let $I \subseteq \bR$ be an open interval and $g: I \to \bR$ be a function.
The \textbf{oscillation} of $g$ at $p \in I$ is defined to be the supremum of all $\epsilon > 0$ such that for all $\delta > 0$ there are $p - \delta < x,y < p + \delta$ such that $|g(x) - g(y)| \geq \epsilon$.
It follows by the triangle inequality that if the oscillation of $g$ at $p$ is at least $\epsilon$ then for all $\delta > 0$ there is a $p - \delta < q < p + \delta$ such that $|g(p) - g(q)| \geq \frac{1}{2}\epsilon$.
It is not difficult to show that, given $\epsilon >0$, the set of points at which $g$ has oscillation at least $\epsilon$ is closed.
We now prove a basic lemma:

\begin{Lem}\label{osc}
Let $g: I \to \bR$ be a definable function.
One of the following holds:
\begin{enumerate}
\item There is an open subinterval of $I$ on which $g$ is continuous.
\item There is an open interval $J \subseteq I$ and an $\epsilon > 0$ such that $g$ has oscillation at least $\epsilon$ at each point in $J$.
\end{enumerate}
\end{Lem}

\begin{proof}
For each $t > 0$ we let $A_t \subseteq I$ be the set of points at which $g$ has oscillation at least $t$.
Note that each $A_t$ is closed.
The set of points at which $g$ is discontinuous is exactly the union of the $A_t$.
If some $A_t$ has nonempty interior then $(2)$ above holds.
Suppose that each $A_t$ has empty interior and is therefore nowhere dense.
As $\{ A_t : t > 0\}$ is a decreasing definable family of sets it follows from Lemma~\ref{family2} that the union of the $A_t$ is nowhere dense.
Then $(1)$ holds.
\end{proof}
\noindent We now prove a weak monotonicity theorem for continuous $\mathfrak{R}$-definable functions.
\begin{Prop}\label{monotone}
Let $I \subseteq \bR$ be an open interval and $f: I \to \bR$ be a nonconstant continuous definable function.
Then there is an open subinterval of $I$ on which $f$ is strictly increasing or strictly decreasing.
\end{Prop}

\begin{proof}
For technical reasons we let $I_0$ be a bounded open subinterval of $I$ whose closure is contained in $I$.
If $f$ is constant on every such interval then $f$ is constant.
We therefore assume that $f$ is nonconstant on $I_0$ and show that there is an open subinterval of $I_0$ on which $f$ is either strictly increasing or strictly decreasing.
It suffices to find an open subinterval of $I_0$ on which $f$ is a homeomorphism.
We assume that $f$ is nonconstant and find a subinterval on which $f$ is a homeomorphism.
Let $P$ be the image of $f$.
Then $P$ is an interval.
As $f$ is nonconstant $P$ has nonempty interior.
Let $g: P \to I_0$ be given by letting $g(p)$ be the minimal element of $f^{-1}(p)$.
This definition makes sense by continuity of $f$.
Then $g$ is an inverse of $f$. It suffices to find an open subinterval of $P$ on which $g$ is continuous.
We suppose towards a contradiction that there is no such subinterval.
Applying Lemma~\ref{osc} we fix an open subinterval $J \subseteq P$ and an $\epsilon > 0$ such that $g$ has oscillation at least $2\epsilon$ at every point in $J$.
Thus for every $p \in J$ and $\delta > 0$ there is a $p - \delta < q < p + \delta$ such that $|f(p) - f(q)| \geq \epsilon$.
As $I_0$ is bounded the image of the restriction of $g$ to $J$ is bounded.
We let $p \in J$ be such that
$$ \sup\{ g(q) : q \in J \} - g(p) < \epsilon. $$
Let $\{x_n\}_{n \in \mathbb{N}}$ be a sequence of elements of $J$ such that $x_n \to p$ as $n \to \infty$ and $|g(x_n) - g(p)| \geq \epsilon$ for every $n$.
This implies that $g(x_n) < g(p)$ and so $g(p) - g(x_n) \geq \epsilon$ for every $n$.
We declare $y_n = g(x_n)$ for all $n$.
As $I_0$ is bounded there is a subsequence of $\{y_n\}_{n \in \mathbb{N}}$ which converges to a point $y$.
As the closure of $I_0$ is contained in $I$ we have $y \in I$.
After restricting to a subsequence if necessary we suppose that $\{y_n\}_{n \in \mathbb{N}}$ converges to $y \in I$.
By definition of $g$ we have $f(y_n) = x_n$ for all $n$.
As $f$ is continuous we have:
$$ f(y) = \lim_{n \to \infty} f(y_n) = \lim_{n \to \infty} x_n = p. $$
As $g(p) - y_n \geq \epsilon$ for all $n$ we have $g(p) - y \geq \epsilon$.
Thus $g(p) > y$.
This is a contradiction as $f(y) = p$ and $g(p)$ is the minimum of $f^{-1}(p)$.
\end{proof}
\noindent The theorem below is an easy consequence of the previous proposition.
\begin{Thm}\label{mono}
 Let $I \subseteq \bR$ be an open interval and let $f: I \to \bR$ be a continuous definable function.
There is a definable open $U \subseteq I$ dense in $I$ such that $f$ is strictly increasing, strictly decreasing, or constant on each connected component of $U$.
\end{Thm}

\begin{proof}
We define several subsets of $I$.
\begin{enumerate}
\item Let $U_1 \subseteq I$ be the set of $p \in I$ such that $f$ is strictly increasing on $(p - \delta, p + \delta)$ for some $\delta > 0$.
\item Let $U_2 \subseteq I$  be the set of $p \in I$ such that $f$ is strictly decreasing on $( p - \delta , p + \delta)$ for some $\delta > 0$.
\item Let $U_3$ be the set of $p \in I$ such that $f$ is constant on $(p - \delta, p + \delta)$ for some $\delta > 0$.
\end{enumerate}
Each $U_i$ is open for $i \in \{1,2,3\}$ by definition.
Proposition~\ref{monotone} implies that every open subinterval of $I$ contains an open interval on which $f$ is strictly increasing, strictly decreasing or constant.
Thus $U_1 \cup U_2 \cup U_3$ is dense in $I$. No connected component of $U_i$ intersects any connected component of $U_j$ for distinct $i,j \in \{1,2,3\}$.
Therefore the connected components of $U_1 \cup U_2 \cup U_3$ are exactly the connected components of $U_1,U_2$ and $U_3$.
It follows that $f$ is strictly increasing, strictly decreasing or constant on each connected component of $U_1 \cup U_2 \cup U_3$.
\end{proof}
\noindent  We now show that continuous definable functions from $\bR$ to $\bR$ are generically locally Lipschitz.
\begin{Prop}\label{lip}
Let $I \subseteq \bR$ be an open interval and let $f
: I \to \bR$ be a nonconstant monotone continuous definable function.
Then $f$ is locally Lipschitz on an open dense subset of $I$.
If $f$ is strictly increasing or strictly decreasing then $f$ is locally bi-Lipschitz on a dense open subset of $I$.s
\end{Prop}
\noindent  We first prove an Lemma which should of independent interest.
\begin{Lem}\label{measure}
A nowhere dense definable subset of $\bR$ is Lebesgue null.
\end{Lem}

\begin{proof}
It suffices to show that a definable closed nowhere dense $A \subseteq \bR$ is Lebesgue null.
We suppose that $A$ has positive measure.
Let $D$ be the set of midpoints of bounded complementary intervals of $A$ and let $\sigma : D \to A$ be given by
$$ \sigma(x) = \min\{ a \in A : x \leq a\}. $$
Then $\sigma(D)$ is the set of right endpoints of bounded complementary intervals of $A$ and is thus dense in $A$.
As $A$ has positive Lebesgue measure it follows from a classical theorem of Steinhaus \cite{Steinhaus} that $\{ x - y : x,y \in A\}$ has nonempty interior.
Thus the set $\{ \sigma(x) - \sigma(y) : x,y \in D\}$ is somewhere dense and so the function $\tau :  D^2 \to \bR$ given by $\tau(x,y) = \sigma(x) - \sigma(y)$ maps $D^2$ onto a somewhere dense subset of $\bR$.
This contradicts Theorem C.
\end{proof}

\noindent We do not know whether in general the statement \emph{``has Lebesgue measure zero''} can be replaced by \emph{``has Hausdorff dimension 0''}. However, if we assume that $\mathfrak R$ defines the function $x\mapsto r x$ for every $r\in \bR$, then indeed every definable subset of $\bR$ has Hausdorff dimension 0. This follows easily from the proof of Lemma \ref{measure} and Edgar and Miller \cite[Lemma 1]{edgarmiller}. Moreover, under this assumption higher dimensional analogues of this result can be obtained by using the techniques from \cite{Forna}. Even in this setting we do not know whether \emph{``has Lebesgue measure zero''} can be replaced by the stronger statement \emph{``has upper Minkowski dimension 0''}. While this strengthening of Lemma \ref{measure} holds for expansions of the real field by \cite{FHM}, we doubt that it holds for expansions of $(\bR,<,+)$ that do not define $\cB$.

\begin{proof}[Proof of Proposition~\ref{lip}]
We first suppose that $f$ is monotone and show that $f$ is locally Lipschitz on a dense open subset of $I$.
For each $\lambda, \delta \in \mathbb{Q}^>$ we let $B_{\lambda, \delta}$ be the set of $x \in I$ such that
\[
|f(x) - f(y)| \leq \lambda |x - y| \quad \text{for all $y\in I$ such that } x - \delta \leq y < x + \delta.
\]
Note that each $B_{\lambda, \delta}$ is definable.
It follows by continuity of $f$ that each $B_{\lambda,\delta}$ is closed.
Let $B$ be the union of the $B_{\lambda,\delta}$.
If $f$ is differentiable at $p$ then $p \in B$.
As $f$ is monotone the Lebesgue differentiability theorem implies that the set of $p \in I$ at which $f$ is not differentiable is Lebesgue null.
Thus $I \setminus B$ is Lebesgue null.
Let
$$U := \bigcup_{\lambda, \delta \in \mathbb{Q}^>} \interior{B_{\lambda,\delta}}.$$
By Lemma~\ref{measure} each $B_{\lambda,\delta} \setminus \interior{B_{\lambda,\delta}}$ is Lebesgue null.
It follows that $I \setminus U$ is Lebesgue null.
In particular $U$ is dense in $I$.
It is easy to see that $f$ is locally Lipschitz on each $\interior{B_{\lambda, \delta}}$, it follows that $f$ is locally Lipschitz on $U$.

We now suppose that $f$ is strictly increasing or strictly decreasing and show that $f$ is locally bi-Lipschitz on a dense open subset of $I$.
Let $J = f(I)$, note that $J$ is an interval.
As $f$ is strictly increasing or strictly decreasing $f$ is a homeomorphism $I \to J$.
Let $g : J \to I$ be the inverse of $f$.
There is a dense open subset $V$ of $J$ on which $g$ is locally Lipschitz.
Then $g(V)$ is open and dense in $I$ as $g$ is a homeomorphism.
It is easy to see that $U \cap g(V)$ is a dense open subset of $I$ on which $f$ is locally bi-Lipschitz.
\end{proof}

\noindent Proposition~\ref{lip} and Theorem~\ref{mono} together yield Theorem E.

\begin{thmE}
Let $\mathcal R$ be an expansion of $(\bR,<,+)$ that does not define $\cB$. Let $f: \bR \to \bR$ be a definable continuous function in $\mathcal R$. Then there is a definable open dense $U \subseteq \bR$ such that $f$ is strictly increasing, strictly decreasing, or constant on each connected component of $U$. Moreover, there is an open dense subset $V$ of $\bR$ such that $f$ is either constant or locally bi-Lipschitz on each connected component of $U$.
\end{thmE}

\bibliographystyle{alpha}
\bibliography{monadicbib}
\end{document}